\documentclass{jloganal}    

\title{Principles of bar induction and continuity on {B}aire space}

\author[Tatsuji Kawai]{Tatsuji Kawai}
\givenname{Tatsuji}
\surname{Kawai}
\address{
Japan Advanced Institute of Science and Technology\\\newline
1-1 Asahidai\\Nomi\\Ishikawa 923-1292\\Japan}
\email{tatsuji.kawai@jaist.ac.jp}
\urladdr{}



\volumenumber{}
\issuenumber{}
\publicationyear{}
\papernumber{}
\startpage{}
\endpage{}
\doi{}
\MR{}
\Zbl{}
\received{}
\revised{}
\accepted{}
\published{}
\publishedonline{}
\proposed{}
\seconded{}
\corresponding{}
\editor{}
\version{}

\usepackage{amssymb, latexsym}

\usepackage[shortlabels]{enumitem}
\setlist[enumerate]{leftmargin=*,align=left,labelindent=\parindent}

\newcommand{\Nat}{\mathbb{N}}
\newcommand{\Bin}{\left\{ 0,1 \right\}}

\newcommand{\cFT}{{\textrm{\textup{c--FT}}}}
\newcommand{\cBI}{{\textrm{\textup{c--BI}}}}
\newcommand{\mBI}{{\textrm{\textup{BI$_{\mathbf{M}}$}}}}
\newcommand{\dBI}{{\textrm{\textup{BI$_{\mathbf{D}}$}}}}
\newcommand{\PiOneBI}{{\textrm{\textup{$\Pi_{1}^{0}$-BI}}}}

\newcommand{\UC}{{\textrm{\textup{UC}}}}
\newcommand{\UCb}{{\textrm{\ensuremath{\textup{BC}}}}}
\newcommand{\UCB}{{\textrm{\ensuremath{\textup{UCB}}}}}

\newcommand{\BCN}{{\textrm{\textup{BC-N}}}}
\newcommand{\CN}{{\textrm{\textup{C-N}}}}
\newcommand{\WCN}{{\textrm{\textup{WC-N}}}}
\newcommand{\PCN}{{\textrm{\textup{PC-N}}}}

\newcommand{\AC}{{\textrm{\ensuremath{\textrm{AC}_{01}}}}}
\newcommand{\ACUni}{{\textrm{\ensuremath{\textrm{AC}_{10}!}}}}
\newcommand{\LLPO}{{\textrm{\textup{LLPO}}}}
\newcommand{\LPO}{{\textrm{\textup{LPO}}}}

\newcommand{\HAw}{{\textrm{\ensuremath{\textrm{HA}^{\omega}}}}}
\newcommand{\IDB}{{\textrm{\ensuremath{\textrm{IDB}_{1}}}}}

\newcommand{\defeqiv}{\mathrel{\stackrel{\textup{def}}{\iff}}}
\newcommand{\defeql}{\mathrel{\stackrel{\textup{def}}{=}}}
\newcommand{\dotminus}{\mathbin{\ooalign{\hss\raise.5ex\hbox{$\cdot$}\hss\crcr$-$}}}

\newcommand{\imp}{\mathrel{\rightarrow}}
\DeclareMathOperator{\sg}{\mathrm{sg}}

\newcommand{\Cons}[2]{#1 * \langle #2 \rangle}
\newcommand{\nil}{\langle \rangle}

\newtheorem{theorem}{Theorem}[section]
\newtheorem{proposition}[theorem]{Proposition}
\newtheorem{lemma}[theorem]{Lemma}

\theoremstyle{definition}

\theoremstyle{remark}
\newtheorem{remark}[theorem]{Remark}

\numberwithin{equation}{section}

\begin{document}

\begin{abstract}
 Brouwer-operations, also known as inductively defined neighbourhood
 functions, provide a good notion of continuity on Baire space
 which naturally extends that of uniform continuity on Cantor space.
 In this paper, we introduce a continuity principle for Baire
 space which says that every pointwise continuous function from Baire
 space to the set of natural numbers is induced by a
 Brouwer-operation.

 Working in Bishop constructive mathematics, we show that the 
 above principle is equivalent to a version of
 bar induction  whose strength is between that of the
 monotone bar induction and the decidable bar induction. We also show
 that the monotone bar induction and the decidable bar induction can
 be characterised by similar principles of continuity.

 Moreover, we show that the $\Pi^{0}_{1}$ bar induction in general
 implies $\LLPO$ (the lesser limited principle of omniscience).  This,
 together with a fact that the $\Sigma^{0}_{1}$ bar induction implies
 $\LPO$ (the limited principle of omniscience), shows that
 an intuitionistically acceptable form of bar induction requires
 the bar to be monotone.

\end{abstract}
\maketitle

\section{Introduction}\label{sec:Introduction}
The uniform continuity principle ($\UC$) is the following statement:
\begin{description}\label{eq:UC}
  \item[\UC] Every pointwise continuous function 
    $F \colon \Bin^{\Nat} \to \Nat$ is uniformly continuous.
\end{description}
In classical mathematics, the above statement is true
because Cantor space $\Bin^{\Nat}$ is topologically compact. This is not the case
in Bishop constructive mathematics \cite{Bishop-67}. In fact, $\UC$
implies the decidable version of Brouwer's fan theorem to which there
is a well-known recursive counterexample (see Troelstra and van Dalen \cite[Chapter 4,
Section 7.6]{ConstMathI}). Here, the fan theorem is a statement saying
that every bar of Cantor space is uniform (see Section
\ref{sec:CBI} for terminology).


The connection between $\UC$ and the fan theorem
is well studied in constructive reverse mathematics
\cite{ConstRevMatheCompactness}.
It is well known that the fan theorem is equivalent to compactness of Cantor
space \cite[Chapter 4, Section 6]{ConstMathI}, and hence it implies $\UC$.
Josef Berger \cite{BergerFANandUC} showed that a weaker version of $\UC$ is
equivalent to the decidable fan theorem (see also Remark
\ref{rem:BergerdFTUC}). In another paper \cite{BergerUCandcFT}, he
also introduced a variant of fan theorem, called $\cFT$, and showed
that it is equivalent to $\UC$. 

In this paper, we establish analogous correspondence between several
notions of continuity on Baire space $\Nat^{\Nat}$ and a variety of bar
induction. 
Our focus is on the relation between various versions of bar induction
and statements similar to $\UC$, but we consider
functions on Baire space instead of Cantor space and replace uniform
continuity with a suitable notion of continuity on Baire space.  More
precisely, we consider a function from $\Nat^{\Nat}$ to $\Nat$ induced
by a \emph{Brouwer-operation} (Kreisel and Troelstra \cite[Section 3]{KreiselTroelstra}) to
be a fundamental notion of continuity on Baire space. The notion can
be considered as a natural generalisation of that of uniform continuity on
Cantor space to the setting of Baire space, since it becomes
equivalent to uniform continuity when restricted to Cantor space (see
Proposition \ref{prop:UniReal}).

We now summarise our main contributions.
First, we formulate a continuity principle for Baire
space, called \emph{the principle of Brouwer continuity} ($\UCb$), based on the notion of Brouwer-operation.  The
principle $\UCb$ states that every pointwise continuous function from
Baire space to the set of natural numbers is induced by a Brouwer-operation.
Then, we introduce a variant of bar induction, called \emph{the
continuous bar induction} ($\cBI$), and show
that $\cBI$ is equivalent to
$\UCb$. Moreover, we characterise the other versions of bar
induction, the monotone bar induction and the decidable bar induction, by
a stronger and a weaker version of $\UCb$ by
varying the strength of the premise of $\UCb$.
Finally, we show that the $\Pi^{0}_{1}$ bar induction (of which $\cBI$
is an instance) in general implies the non-constructive principle $\LLPO$
(the lesser limited principle of omniscience), and thus intuitionistically
unacceptable.  

The relation between several versions of bar induction and continuity
axioms (namely strong and weak continuity for numbers, and bar
continuity) has been extensively studied by Howard and Kreisel
\cite{HowardKreisel} and Kreisel and Troelstra
\cite{KreiselTroelstra}. Some of their results are recalled as
corollaries of our work in Section \ref{sec:PCN} (Theorem
\ref{thm:BarCont}). Our main contribution is in introducing the bar
induction $\cBI$ which is equivalent to $\UCb$ and characterising
the other versions of bar induction by similar principles of
continuity. In this way, the difference between various versions of bar
induction can be understood as the difference between the notions of
continuity involved in the corresponding principles of continuity.

\subsubsection*{Formal system}
We work in Bishop constructive mathematics \cite{Bishop-67}.
However, our work should be formalisable in a suitable extension of
intuitionistic arithmetic in all finite types ($\HAw$), which we
now briefly describe.

First, the language of $\HAw$ is extended with the types of boolean
$\Bin$ and finite sequences $\Bin^{*}$ and $\Nat^{*}$ of $\Bin$ and
$\Nat$ respectively, together with appropriate constructors and axioms
for these types.
Second, we assume the following choice axioms:
\begin{description}
  \item[\AC] $\left( \forall x \in \Nat \right)\left( \exists \alpha
    \in \Nat^{\Nat} \right) A(x,\alpha) \imp \left( \exists F \in
    (\Nat^{\Nat})^{\Nat}\right)\left( \forall x \in \Nat \right) A(x,
    F(x))$.

  \item[\ACUni] 
    $\left( \forall\alpha \in \Nat^{\Nat} \right)\left( \forall x \in
    \Nat \right) \neg B(\alpha,x) \vee B(\alpha,x) \\
    \imp
    \left[
    \left( \forall \alpha \in \Nat^{\Nat} \right)\left(
    \exists ! x  \in \Nat \right) B(\alpha, x)
    \imp 
    \left( \exists F \in \Nat^{(\Nat^{\Nat})} \right)
    \left( \forall \alpha \in \Nat^{\Nat} \right) 
     B(\alpha, F(\alpha)) \right]$.
\end{description}
Moreover, we add a predicate symbol $K$ on $\Nat^{\Nat^{*}}$ together with 
the following axioms (for the notation used, see the next subsection):
\begin{enumerate}[{K}1]
  \item\label{K1} $\lambda a. x + 1 \in K$,

  \item\label{K2} $\left[\alpha(\langle \rangle) = 0 \wedge \left( \forall x \in
    \Nat \right) \lambda a. \alpha(\langle x \rangle * a) \in K
  \right] \imp \alpha \in K$,

  \item\label{K3} $\left( \forall \alpha \in \Nat^{\Nat^{*}} \right)
    \left[ A(Q,\alpha) \imp Q(\alpha) \right] \imp K \subseteq Q$,
\end{enumerate}
where
\[
  A(Q,\alpha) \defeqiv \left( \exists x \in \Nat \right)
  \left[ \alpha = \lambda a. x+1 \right] \vee \left[ \alpha(\langle
  \rangle) = 0 \wedge
  \left( \forall x \in \Nat \right)\lambda a.\alpha(\langle x \rangle * a)
   \in Q\right].
\]
The predicate $K$ can be understood as being inductively defined
by \ref{K1} and \ref{K2}.

The system described above can be thought of as an extension
of the intuitionistic theory of analysis $\IDB$
described in Kreisel and Troelstra \cite{KreiselTroelstra} to all finite types, together
with the axiom of unique choice $\ACUni$. See
Troelstra and van Dalen \cite{ConstMathI,ConstMathII} for the details of the systems $\HAw$
and $\IDB$. 

\subsubsection*{Notation}
We adopt the following notation in this paper.
The letters $k,n,m,x,y$ range over natural numbers $\Nat$.
The letters $a,b$ range over the finite
sequences $\Nat^{*}$ of natural numbers or the finite binary sequences
$\Bin^{*}$.  Greek letters $\alpha,\beta,\gamma,\dots$ range over the
infinite sequences $\Nat^{\Nat}$  or $\Bin^{\Nat}$. 
We write $|a|$ for the length of $a$ and $a * b$ for the
concatenation of $a$ and $b$.
We write $\nil$  and $\langle n \rangle$ for the empty sequence
and a sequence of length $1$. We write $a \preccurlyeq b$ to
mean that $a$ is an initial segment of $b$.
Moreover, we write $\overline{\alpha}k$ for the initial segment of $\alpha$
of length $k$, and we let $\alpha \in a$  abbreviate
$\overline{\alpha}|a| = a$. 
We extend concatenation between finite sequences to the one between
finite sequences and infinite sequences by letting $a * \alpha$
denote the sequence
such that $a * \alpha \in a$ and $\left( \forall n \in \Nat\right)
n \geq |a| \imp a * \alpha(n) = \alpha(n \dotminus |a|)$.

We let $A,B,C,\dots$ range over the formulas of our system.
By a predicate of type $\mathbb{T}$, we mean a formula $A$ of our
system with a free variable of type $\mathbb{T}$. In this case, we write $A \subseteq
\mathbb{T}$. For predicates $A,B \subseteq \mathbb{T}$, we let $A
\subseteq B$ abbreviate $\left( \forall t \in \mathbb{T} \right) A(t)
\imp B(t)$. We sometimes write $t \in A$ for $A(t)$.

\section{Continuous bar induction}\label{sec:CBI}
We introduce the principle $\cBI$, the continuous bar
induction, and argue that $\cBI$ naturally extends the fan theorem
$\cFT$ by Berger \cite{BergerUCandcFT}.

A predicate $P \subseteq \Nat^{*}$ is a \emph{bar} if
\begin{equation*}
\left(\forall \alpha \in \Nat^{\Nat}  \right) 
\left( \exists n \in \Nat \right) P(\overline{\alpha}n).
\end{equation*}
A bar $P$ is a \emph{c--bar} if there exists a function $\delta \colon
\Nat^{*} \to \Nat$ such that 
    \[
    \left(\forall a \in \Nat^{*}  \right)
    \left[P(a) \leftrightarrow 
    \left( \forall b \in \Nat^{*} \right) \delta(a) = \delta(a * b)
    \right].
    \]
A predicate $Q \subseteq \Nat^{*}$ is \emph{inductive} if 
\[
  \left( \forall a \in \Nat^{*} \right) \left[ \left( \forall n \in \Nat
  \right) Q(a * \langle n \rangle) \rightarrow Q (a) \right].
\]

\emph{The continuous bar induction} ($\cBI$) is the following statement:
\begin{description}
  \item[\cBI] For any c--bar $P \subseteq \Nat^{*}$ and a predicate $Q \subseteq
    \Nat^{*}$, if $P \subseteq Q$ and $Q$ is inductive, then
    $Q(\langle \rangle)$.
\end{description}

%
%
In the rest of this section, we relate $\cBI$ to the fan theorem
$\cFT$.

We recall the standard terminology.
If $P$ and $Q$ are predicates of some type $\mathbb{T}$
such that $P \subseteq Q$,
we say that $P$ is \emph{detachable} from $Q$ if
\[
  \left( \forall t \in \mathbb{T} \right) Q(t) \imp \neg P(t) \vee P(t).
\]
A predicate $C \subseteq \Bin^{*}$ is a \emph{c--set} if there exists a
detachable predicate $D \subseteq \Bin^{*}$ such that
\[
  \left( \forall a \in \Bin^{*} \right)\Bigl[ C(a) \leftrightarrow \left(
  \forall b
  \in \Bin^{*} \right) D(a * b)  \Bigr].
\]
A predicate $P \subseteq \Bin^{*}$ is a \emph{bar} of the binary tree $\Bin^{*}$ if
\[
  \left( \forall \alpha \in \Bin^{\Nat} \right)\left( \exists n \in
  \Nat \right)P(\overline{\alpha}n).
\]
A bar $P \subseteq \Bin^{*}$ is \emph{uniform} if
\[
  \left( \exists N \in \Nat \right)\left( \forall \alpha \in
  \Bin^{\Nat} \right)\left( \exists n \leq N \right)
  P(\overline{\alpha}n).
\]
The principle $\cFT$ is the following statement \cite{BergerUCandcFT}:
\begin{description}
  \item[\cFT] Every bar $P \subseteq \Bin^{*}$ that is a c--set is
    uniform. 
\end{description}

\begin{proposition}
  \leavevmode
  \begin{enumerate}
    \item\label{prop:cbar}
      Let $P \subseteq \Bin^{*}$ be a bar for the binary
      sequences. Then, $P$ is a c--set
      if and only if there exists a function $\delta \colon \Bin^{*}
      \to \Nat$ such that
      \begin{equation}\label{prop:eq:cbar}
        \left( \forall a \in \Bin^{*} \right) \Bigl[ P(a) \leftrightarrow
        \left( \forall b \in \Bin^{*} \right) \delta(a) =
        \delta(a*b) \Bigr].
      \end{equation}
    \item\label{prop:cBIcFT} $\cBI \implies \cFT$.
  \end{enumerate}
\end{proposition}
\begin{proof}
  \eqref{prop:cbar}
  ($\Rightarrow$)
  Suppose that $P$ is a c-set that is a bar. Let $D \subseteq
  \Bin^{*}$ be a detachable predicate such that
  \[
    \left( \forall a \in \Bin^{*} \right)
    \left[ P(a) \leftrightarrow \left( \forall b \in \Bin^{*} \right)
    D(a*b)\right] .
  \]
  Define a function $\delta \colon \Bin^{*} \to \Nat$ by 
  \[
    \delta(a) \defeql
    \begin{cases}
      1 &\text{if $D(a)$},\\
      0 &\text{otherwise}.
    \end{cases}
  \]
  Let $a \in \Bin^{*}$, and suppose that $\left( \forall b \in \Bin^{*} \right)
  \delta(a) = \delta(a * b)$.
  Since $P$ is a bar, there exists $n \in \Nat$ such that
  $P(\overline{a*0^{\omega}}n)$, where 
 \[
   0^{\omega} \defeql \lambda x. 0. 
 \]
  Then, obviously 
  $\delta(a) = \delta(\overline{a*0^{\omega}}n) = 1$.
  Hence $P(a)$.  The converse $P(a) \imp \left( \forall b  \in \Bin^{*} \right)
  \delta(a) = \delta(a * b)$  is obvious.

  ($\Leftarrow$)
  Let $\delta \colon \Bin^{*} \to \Nat$
  be a function that satisfies the condition \eqref{prop:eq:cbar}. Define a detachable predicate
  $D \subseteq  \Bin^{*}$ by
  \[
    D(a) \defeqiv \delta(a) = \delta(\Cons{a}{0}) = \delta(\Cons{a}{1}).
  \]
  Obviously we have $P(a) \leftrightarrow \left( \forall b \in \Bin^{*} \right) D(a*b)$.

  \eqref{prop:cBIcFT} Assume $\cBI$.
  Let $C \subseteq \Bin^{*}$ be a c-set which is a bar of the binary
  tree.
  By the first part of this proposition, there exists a function $\delta \colon \Bin^{*} \to \Nat$
  such that 
  \[
    C(a) \leftrightarrow \left( \forall b \in \Bin^{*} \right)
    \delta(a) = \delta(a * b).
  \]
  Define a function $\Gamma \colon \Nat^{\Nat} \to \Bin^{\Nat}$ by
  \[
    \Gamma(\alpha) \defeql \lambda n. \sg(\alpha(n)),
  \]
  where $\sg(n) \defeql \min(1,n)$.
  Similarly, we define $\Gamma^{*} \colon \Nat^{*} \to \Bin^{*}$.
  Since $C$ is a bar, we have
  $\left( \forall \alpha \in \Nat^{\Nat} \right)\left(
  \exists n \in \Nat \right) C(\overline{\Gamma(\alpha)}n)$, i.e.\
  $\left( \forall \alpha \in \Nat^{\Nat} \right)\left(
  \exists n \in \Nat \right) C(\Gamma^{*}(\overline{\alpha}n))$.
  Define a predicate $P \subseteq \Nat^{*}$ and a function $\varepsilon \colon
  \Nat^{*} \to \Nat$  by
  \begin{align*}
    P(a) &\defeqiv C(\Gamma^{*}(a)), \\
    \varepsilon(a) &\defeql \delta(\Gamma^{*}(a)).
  \end{align*}
  Then, $P(a) \leftrightarrow \left( \forall b \in \Nat^{*}\right)
  \varepsilon(a) = \varepsilon(a * b)$, so $P$ is a c--bar.
  Define a predicate $Q \subseteq \Nat^{*}$ by
  \[
    Q(a) \defeqiv \left( \exists N \in \Nat \right)\left( \forall
    \alpha \in \Nat^{\Nat}\right)\left( \exists n \leq N \right) P(a *
    \overline{\alpha} n).
  \]
  Clearly, $P \subseteq Q$. Let $a \in \Nat^{*}$ and suppose that
  $\left( \forall n \in \Nat \right)Q(\Cons{a}{n})$.
  Then, there exists $N \in \Nat$ such that for each $i \in \{0,1\}$, 
  \[
    \left( \forall \alpha \in \Nat^{\Nat}\right) \left( \exists n \leq
    N \right) P(\Cons{a}{i} * \overline{\alpha} n). 
  \]
  From the definition of $P$, we see that $Q(a)$. Thus, $Q$ is
  inductive. Applying $\cBI$, we obtain $Q(\nil)$, which implies
  \[
    \left( \exists N \in \Nat \right)\left( \forall \alpha \in
    \Bin^{\Nat}\right) \left( \exists n \leq
    N \right) C(\overline{\alpha} n). \qedhere
  \]
\end{proof}
Thus, we can think of $\cBI$ as a generalisation of $\cFT$ to
Baire space.
\section{The principle of Brouwer continuity}\label{sec:UCb}
We recall the notion of Brouwer-operation
from Kreisel and Troelstra \cite{KreiselTroelstra}, which allows us to give a constructive
notion of continuity on Baire space which naturally extends
the notion of uniform continuity on Cantor space.

The predicate $K \subseteq \Nat^{\Nat^{*}}$ of \emph{Brouwer-operations}
is inductively defined by the following clauses:
\begin{gather}\label{def:K}
  \frac{n \in \Nat}{\lambda a. n + 1 \in K},  \qquad
  \frac{\gamma(\langle \rangle)  = 0 \quad \left(\forall n \in \Nat  \right) \lambda a.
  \gamma(\langle n \rangle * a) \in K}{\gamma \in K}.
\end{gather}
Formally, we assume the existence of a predicate $K$ satisfying the
axioms \ref{K1} -- \ref{K3}; see Introduction \ref{sec:Introduction}.
If a Brouwer-operation $\gamma \in K$ is introduced by the second
clause, we write $\sup_{n \in \Nat} \gamma_{n}$ for $\gamma$, where
\[
  \gamma_{n} \defeql \lambda a. \gamma(\langle n \rangle * a).
\]


Let $K_{0}$ be a predicate on $\Nat^{\Nat^{*}}$ defined by 
\begin{multline*}
   K_{0}(\gamma) \defeqiv \left( \forall \alpha \in \Nat^{\Nat}
  \right)\left( \exists n \in \Nat \right) \gamma(\overline{\alpha}n)
  > 0 \mathrel{\wedge}\\
  \left( \forall a,b \in \Nat^{*} \right)\bigl[ \gamma(a) > 0 \rightarrow
  \gamma(a) = \gamma(a*b)  \bigr].
\end{multline*}
An element of $K_{0}$ is called a \emph{neighbourhood
function}. Note that every Brouwer-operation is a neighbourhood function.
\begin{lemma}\label{lem:BwOpisNhood}
  $K \subseteq K_{0}$.
\end{lemma}
\begin{proof}
  The proof is by induction on $K$.
  Details can be found in Troelstra and van Dalen \cite[Chapter 4, Proposition 8.5]{ConstMathI}.
\end{proof}
The converse of Lemma \ref{lem:BwOpisNhood} does not necessarily
hold; see Lemma \ref{lem:KK0equivdBI}.

By $\ACUni$, every neighbourhood function $\gamma \in K_{0}$
determines a function
$F_{\gamma} \colon \Nat^{\Nat} \to \Nat$ by
\begin{equation}\label{eq:BrouwrCont}
  F_{\gamma}(\alpha) \defeql \gamma\left(
  \overline{\alpha}\min_{z \in \Nat}\left[\gamma(\overline{\alpha}z) > 0
  \right] \right) \dotminus 1.
\end{equation}
A function $F \colon \Nat^{\Nat} \to \Nat$ is
\emph{$K_{0}$-realisable} if there exists a neighbourhood function
$\gamma \in K_{0}$ such that $F_{\gamma} = F$.  Similarly, a function $F
\colon \Nat^{\Nat} \to \Nat$ is said to be \emph{$K$-realisable} if there exists
a Brouwer-operation $\gamma \in K$ such that $F_{\gamma} = F$.  In
both cases, we say that $\gamma$ \emph{realises} $F$ and write $\gamma
\Vdash F$.

We now formulate a continuity principle for Baire space, called
\emph{the principle of Brouwer continuity}
($\UCb$):\footnote{The principle $\UCb$ is called $\UCB$ in
  \cite{KawaiFContOnBaire}.}
\begin{description}
  \item[\UCb] Every pointwise continuous function 
    $F \colon \Nat^{\Nat} \to \Nat$ is $K$-realisable.
\end{description}
Here, recall that a function $F \colon \Nat^{\Nat} \to \Nat$ is pointwise continuous if
\[
  \left( \forall \alpha \in \Nat^{\Nat} \right)\left( \exists n \in
  \Nat \right) \left( \forall \beta \in \Nat^{\Nat} \right)
  \overline{\beta}n  = \overline{\alpha}n  \imp F(\beta) = F(\alpha).
\]
The following argument highlights the difference between pointwise
continuity and realisability by neighbourhood functions.
Let $K_{1}$ be a predicate on $\Nat^{\Nat^{*}}$ defined by 
\[
  K_{1}(\delta)
  \defeqiv
  \left( \forall \alpha \in \Nat^{\Nat} \right)
  \left( \exists n \in \Nat \right)
  \left( \forall a \in \Nat^{*} \right)
  \delta(\overline{\alpha}n) = \delta(\overline{\alpha}n * a).
\]
Note that $K_{0} \subseteq K_{1}$, and the predicate $K_{1}$ represents
the class of c--bars. 

Every function $\delta \in K_{1}$ determines a pointwise continuous
function $F_{\delta} \colon \Nat^{\Nat} \to \Nat$ in the following
way.
For each $\alpha \in \Nat^{\Nat}$,  define
\begin{equation}\label{eq:D}
   D_{\alpha} \defeql \left\{ m \in \Nat \mid
     \delta(\overline{\alpha}m) \neq
     \delta(\overline{\alpha}(m+1))\right\} \cup \{1\}.
\end{equation}
Then, $D_{\alpha}$ is bounded because $\delta$ determines a c--bar.
By $\ACUni$, defined a function $F_{\delta} \colon \Nat^{\Nat} \to \Nat$
by
\begin{equation*}\label{eq:G}
  F_{\delta}(\alpha) \defeql \delta(\overline{\alpha}(\max D_{\alpha} + 1)).
\end{equation*}
To see that $F_{\delta}$ is pointwise continuous, let $\alpha \in
\Nat^{\Nat}$. Then, there exists $n \in \Nat$ such that $\left
( \forall a \in \Nat^{*} \right)\delta(\overline{\alpha}n) =
\delta(\overline{\alpha}n * a)$.  Then, for any $\beta \in
\overline{\alpha}n$, we have $D_{\beta} = D_{\alpha}$, and so
$F_{\delta}(\alpha) = F_{\delta}(\beta)$. Hence $F_{\delta}$ is
pointwise continuous. Conversely, every  pointwise continuous function
$F \colon \Nat^{\Nat} \to \Nat$ arises in this way from a function $\delta
\in K_{1}$ by setting
\[
  \delta(a) \defeql F(a * 0^{\omega}).
\]

In the rest of this section, we relate $\UCb$ to the uniform continuity
principle $\UC$.

First, we adjust the notion of Brouwer-operation to the functions on
Cantor space.
The predicate $K_{C} \subseteq \Nat^{\Bin^{*}}$ of
\emph{Brouwer-operations} on Cantor space is inductively defined by
the following clauses:
\begin{gather*}
  \frac{n \in \Nat}{\lambda a. n + 1 \in K_C}, \qquad
  \frac{\gamma(\langle \rangle)  = 0 \quad \left(\forall i \in
    \Bin\right) \lambda a.
  \gamma(\langle i \rangle * a) \in K_C}{\gamma \in K_C}.
\end{gather*}
Each Brouwer-operation $\gamma \in K_C$ determines a continuous
function $F_{\gamma} \colon \Bin^{\Nat} \to \Nat$ as in
equation
\eqref{eq:BrouwrCont}. The notion of $K_{C}$-realisable function from
$\Bin^{\Nat}$ to $\Nat$ is similarly defined.

In the following proposition, recall that a function $F \colon
\Bin^{\Nat} \to \Nat$ is \emph{uniformly continuous} if
\begin{equation}\label{def:UCCantor}
  \left( \exists n \in \Nat \right) \left( \forall \alpha,\beta \in
  \Bin^{\Nat} \right)\overline{\alpha} n = \overline{\beta}n \imp F(\alpha) =
  F(\beta).
\end{equation}
\begin{proposition}\label{prop:UniReal}
  A function $F \colon \Bin^{\Nat} \to \Nat$ is uniformly continuous if and only if
  $F$ is $K_{C}$-realisable.
\end{proposition}
\begin{proof}
  ($\Rightarrow$)
  Define a predicate $A \subseteq \Nat$ by
  \begin{multline*}
    A(n) \defeqiv  \left(\forall F \in \Nat^{\left(\Bin^{\Nat}\right)} \right)
    \Bigl[ \left( \forall \alpha,\beta \in \Bin^{\Nat} \right)
      \left[ 
      \overline{\alpha}n = \overline{\beta}n \imp F(\alpha) = F(\beta)  \right] \Bigr.\\
    \Bigl. \imp \left(
    \exists \gamma \in K_{C} \right) \gamma \Vdash F \Bigr].
  \end{multline*}
  It suffices to show that $A(n)$ for all $n \in \Nat$, which is proved by induction.

  $n = 0$: Then $\lambda a. F(0^{\omega}) + 1$ realises $F$.

  $n = k + 1$: Let $F \colon \Bin^{\Nat} \to \Nat$ be a function such that
  \[
    \left(  \forall \alpha,\beta \in \Bin^{\Nat} \right) \overline{\alpha}n =
    \overline{\beta}n \imp
    F(\alpha) = F(\beta).
  \]
  By induction hypothesis, there exist
  $\gamma_0, \gamma_1 \in K_{C}$ such that for each $i \in \Bin$
  the Brouwer-operation $\gamma_i$ realises a function $F_i \colon \Bin^{\Nat} \to \Nat$
  given by 
  \begin{equation}\label{eq:Fi}
    F_i(\alpha) \defeql F(\langle i \rangle * \alpha).
  \end{equation}
  Define $\gamma \in K_{C}$ by
  $\gamma(\nil) \defeql 0$, and $\lambda a. \gamma(\langle i \rangle *
  a) \defeql
  \gamma_{i}$ for $i \in \Bin$. Let $\alpha \in \Bin^{\Nat}$,
  and put $i = \alpha(0)$. Since $\gamma_{i} \Vdash F_i$, there exists
  $k \in \Nat$ such that $\gamma_{i}(\overline{\alpha_{\geq 1}}k) =
  F_{i}(\alpha_{\geq 1}) + 1$, where $\alpha_{\geq 1} \defeql \lambda
  n. \alpha(n + 1)$. Then $\gamma(\overline{\alpha}(k+1)) = F(\alpha)
  + 1$. Therefore $\gamma$ realises $F$.

  ($\Leftarrow$)
  Suppose that $F$ is realised by $\gamma \in K_{C}$.
  We show by induction on $K_{C}$ that
  \[
    \left( \forall F \in \Nat^{\Bin^{\Nat}} \right) \gamma \Vdash F \imp
    \text{``$F$ is uniformly continuous''},
  \]
  where ``$F$ is uniformly continuous'' is the formula of the form
  \eqref{def:UCCantor}.

  $\gamma = \lambda a. n + 1$: Then $\gamma$ realises the constant
  function $\lambda \alpha. n$, which is uniformly continuous.

  $\gamma(\langle \rangle) = 0 \wedge \left( \forall i  \in \Bin
  \right) \lambda a. \gamma(\langle i \rangle * a) \in K_{C}$:
  Let $F \colon \Bin^{\Nat} \to \Nat$ be a function such that
  $\gamma\Vdash F$. Then, for each $i \in \Bin$ we have 
  $\lambda a. \gamma(\langle i \rangle * a) \Vdash F_i$,
  where $F_i$ is defined as in \eqref{eq:Fi}.
  By induction hypothesis, $F_i$ is uniformly continuous for each $i
  \in \Bin$. Hence $F$ is uniformly continuous.
\end{proof}
\section{Equivalence of $\cBI$ and $\UCb$}\label{sec:EeqivcBI-UCb}
The aim of this section is to prove the following equivalence.
\begin{theorem}\label{thm:EquivcBIUCb}
$\cBI \iff \UCb$.
\end{theorem}

First, we prove the direction ($\Rightarrow$).
\begin{proposition}
$\cBI \implies \UCb$.
\end{proposition}
\begin{proof}
  Assume $\cBI$. Let $F \colon \Nat^{\Nat} \to \Nat$ be a pointwise continuous
  function. Define a function $\delta \colon \Nat^{\Nat} \to \Nat$
  and a predicate $P \subseteq \Nat^{*}$ by
  \begin{align*}
    \delta(a) &\defeql F(a * 0^{\omega}),\\
    P(a) &\defeqiv \left( \forall b \in \Nat^{*} \right)
    \delta(a) = \delta(a * b).
  \end{align*}
  Since $F$ is pointwise continuous, $P$ is a c--bar.
  Define a predicate $Q  \subseteq \Nat^{*}$ by 
  \begin{equation}\label{prop:Q_cBIUCb}
    Q(a) \defeqiv \left( \exists \gamma \in K \right) 
    \left( \forall b\in \Nat^{*} \right) \gamma(b) > 0 \imp P(a*b)
    \wedge \gamma(b) = \delta(a*b) + 1.
  \end{equation}
  We show that
  \begin{enumerate}
    \item\label{prop:PQ_cBIUCb} $P \subseteq Q$,
    \item\label{prop:QIND_cBIUCb} $Q$ is inductive.
  \end{enumerate}

  \eqref{prop:PQ_cBIUCb}  Let $a \in \Nat^{*}$ such  that $P(a)$. Define $\gamma \in K$ by
  $\gamma \defeql \lambda b. \delta(a) + 1$. Then, $\gamma$ is a witness of
  the existential quantifier in \eqref{prop:Q_cBIUCb}.  Thus $Q(a)$.

  \eqref{prop:QIND_cBIUCb}  Let $a \in \Nat^{*}$ and suppose that
  $\left( \forall n \in \Nat \right) Q(\Cons{a}{n})$.
  By $\AC$, there exists a sequence $\left( \gamma_{n}
  \right)_{n \in \Nat}$ of Brouwer-operations such that
  \[
    \left( \forall n \in \Nat \right)\left( \forall b \in
    \Nat^{*} \right)
    \gamma_{n}(b) > 0 \imp P(\Cons{a}{n}* b)  \wedge
    \gamma_{n}(b) = \delta(\Cons{a}{n}*b) + 1.
  \]
  Put $\gamma \defeql \sup_{n \in \Nat}\gamma_{n}$. 
  Let $b \in \Nat^{*}$, and suppose that
  $\gamma(b) > 0$. Then, there exist $n \in \Nat$
  and $b' \in \Nat^{*}$ such that $b = \langle n \rangle * b' \wedge \gamma_{n}(b') > 0$.
  Thus, $P(\Cons{a}{n}*b') \wedge \gamma_{n}(b') = \delta(\Cons{a}{n}*b') + 1$,
  that is $P(a * b) \wedge \gamma(b) = \delta(a * b) + 1$.  Hence $Q(a)$.

  By $\cBI$, we obtain $Q(\nil)$, i.e.\ there exists $\gamma \in K$
  such that
  \[
    \left( \forall a \in \Nat^{*} \right)\gamma(a) > 0 \imp
    P(a) \wedge \gamma(a) = \delta(a) + 1.
  \]
  Therefore $\gamma$ realises $F$.
\end{proof}

To prove the  direction ($\Leftarrow$) of Theorem \ref{thm:EquivcBIUCb}, we need some
preliminaries.
\begin{lemma}[{Kreisel and Troelstra \cite[Theorem 3.1.2]{KreiselTroelstra}}]\label{prop:UnSec}
  Let $Q$ be a predicate on $\Nat^{*}$. Then,
  \[
    \left( \forall \gamma \in K \right)
    \Bigl[ P_{\gamma} \subseteq Q  \wedge \left(\forall a \in
    \Nat^{*} \right)\left[ \left(\forall n \in \Nat\right) Q(a*\langle n \rangle) \imp Q(a) \right] \imp Q(\langle \rangle)
    \Bigr],
  \]
  where $P_{\gamma} \defeql  \left\{ a \in \Nat^{*} \mid \gamma(a) > 0 \right\}$.
\end{lemma}
\begin{proof}
  See Kreisel and Troelstra \cite[Theorem 3.1.2]{KreiselTroelstra}.
\end{proof}

We prove the following two lemmas for the sake of completeness.
\begin{lemma}[{Troelstra and van Dalen \cite[Exercise 4.8.5]{ConstMathI}}] \label{lem:KK0}
  \[
    \left( \forall \gamma \in K \right)\left( \forall \gamma' \in K_0 \right)
    \Bigl[ \left( \forall a \in \Nat^{*}\right) \bigl[ \gamma(a) > 0 \imp \gamma'(a) >
    0 \bigr] \imp \gamma' \in K \Bigr].
  \]
\end{lemma}
\begin{proof}
  By induction on $K$.

  $\gamma = \lambda a. n + 1$: For any $\gamma' \in K_0$, if
   $\left( \forall a \in \Nat^{*} \right) \gamma(a) > 0 \imp
   \gamma'(a) > 0$, then $\gamma'$ is a constant function with
   a positive value. Thus $\gamma' \in K$.

  $\gamma = \sup_{n \in \Nat}\gamma_{n}$:
  Let $\gamma' \in K_0$ and suppose that
   $\left( \forall a \in \Nat^{*} \right) \gamma(a) > 0 \imp
   \gamma'(a) > 0$. Then, for each $n \in \Nat$, we have
   $\gamma'_{n} \defeql \lambda a. \gamma'(\langle n \rangle * a) \in
   K_0$ and $\left( \forall a \in \Nat^{*} \right) \gamma_{n}(a) > 0 \imp
   \gamma'_{n}(a) > 0$. By induction hypothesis, we have
   $\gamma'_{n} \in K$ for all $n \in \Nat$. Since
   $\gamma' = \sup_{n \in \Nat} \gamma'_{n}$, we conclude $\gamma' \in K$.
\end{proof}

\begin{lemma}[{Troelstra and van Dalen \cite[Exercises 4.8.6]{ConstMathI}}]\label{lem:keylem}
  \[
    \left( \forall \gamma \in K \right) \lambda a. \gamma(a) \cdot
    \sg(|a| \dotminus \gamma(a)) \in K.
  \]
\end{lemma}
\begin{proof}
  By induction on $K$. Put $\gamma' \defeql \lambda a. \gamma(a) \cdot
    \sg(|a| \dotminus \gamma(a))$.

  $\gamma = \lambda a. n + 1$: This follows from Kreisel and Troelstra
  \cite[Theorem 3.2.2 (iv), (vi)]{KreiselTroelstra}.
  Alternatively, it is clear that 
  $\lambda a. (n+1) \cdot \sg(|a| \dotminus (n+1))$ is introduced in
  $K$ by $n + 2$ -times application of the second clause of \eqref{def:K}.

  $\gamma = \sup_{n \in \Nat}\gamma_{n}$:
  By induction hypothesis, we have $\gamma_{n}' \in K$ for all
  $n \in \Nat$. Put $\xi \defeql \sup_{n \in \Nat}
  \gamma_{n}' \in K$.  Let $a \in \Nat^{*}$ and suppose that $\xi(a) >
  0$. Then, there exists $n \in \Nat$ and $a' \in \Nat^{*}$ such that $a =
  \langle n \rangle * a'$ and $\gamma_{n}'(a') > 0$.
  Thus $\gamma'(a) = \gamma(a) \cdot \sg(|a| \dotminus \gamma(a)) >
  0$. Clearly, we have $\gamma' \in K_{0}$. Hence
  $\gamma' \in K$ by Lemma \ref{lem:KK0}.
\end{proof}

We now prove the direction ($\Leftarrow$) of Theorem
\ref{thm:EquivcBIUCb}.
\begin{proposition}\label{prop:UCbimpcBI}
$\UCb \implies \cBI$.
\end{proposition}
\begin{proof}
  Let $P \subseteq \Nat^{*}$ be a bar, and let $\delta \colon
  \Nat^{*} \to \Nat$ be a function such that $P(a) \leftrightarrow \left( \forall b
  \in \Nat^{*} \right) \delta(a) = \delta(a*b)$. Let $Q \subseteq
  \Nat^{*}$ be an inductive predicate such that $P \subseteq Q$. 
  Define a function $F \colon \Nat^{\Nat} \to \Nat$ by
  \[
    F(\alpha) \defeql \max D_{\alpha},
  \]
  where $D_{\alpha}$ is given by the equation \eqref{eq:D}. Then, $F$ is pointwise
  continuous. By $\UCb$, there exists
  a Brouwer-operation $\gamma \in K$ such that $F_{\gamma} = F$.  By Lemma
  \ref{lem:keylem}, we may assume that 
    $
    \left( \forall a \in \Nat^{*}
    \right) \gamma(a) > 0 \rightarrow |a| > \gamma(a). 
    $
  Let $a \in \Nat^{*}$ such that $\gamma(a) > 0$.
  Let $b \in \Nat^{*}$. Then,
  $|a| > \gamma(a) \wedge \gamma(a) = \gamma(a * b)$ so that
  \[
    |a| > \max D_{a * 0^{\omega}} + 1 = \max D_{a * b * 0^{\omega}} + 1.
  \]
  Thus $\delta(a) = \delta(a * b)$. Hence $P(a)$, and so $Q(a)$.
  By Proposition \ref{prop:UnSec}, we obtain $Q(\nil)$.
\end{proof}
This completes the proof of Theorem \ref{thm:EquivcBIUCb}.
We note that the structure of the proof of 
Proposition \ref{prop:UCbimpcBI} is quite similar to the proof of the
implication $\UC \implies \cFT$ by Berger \cite[Proposition 2]{BergerUCandcFT}.

\section{Characterisation of bar inductions by 
continuity principles}\label{sec:BIandUC}
We show that the decidable bar induction and the monotone bar
induction can be characterised by statements similar to $\UCb$.

\subsection{Decidable bar induction}
The decidable bar induction $\dBI$ is the following statement:
\begin{description}
  \item[\dBI] For any detachable bar $P \subseteq \Nat^{*}$ and a
    predicate $Q \subseteq \Nat^{*}$, if $P \subseteq Q$ and $Q$ is
    inductive, then $Q(\langle \rangle)$.
\end{description}

We relate $\dBI$ to two notions of continuity.

First, recall that in Section \ref{sec:UCb} we defined a function $F \colon \Nat^{\Nat}
\to \Nat$ to be \emph{$K_{0}$-realisable} if there exists a neighbourhood
function $\gamma \in K_0$ such that $F_{\gamma} = F$.

Next, given a function $F \colon \Nat^{\Nat} \to \Nat$, a function $g \colon
\Nat^{\Nat} \to \Nat$ is a \emph{modulus} of $F$ if 
\begin{equation}\label{eq:modulus}
  \left( \forall \alpha \in \Nat^{\Nat} \right)\left( \forall \beta
  \in \overline{\alpha}g(\alpha) \right) F(\beta) = F(\alpha).
\end{equation}

The following lemma is due to Beeson \cite[Chapter VI, Section 8,
Exercise 8]{BeesonFoundationConstMath}.\footnote{In 
Beeson \cite{BeesonFoundationConstMath}, neighbourhood functions are
called \emph{associates}.}
\begin{lemma}
  \label{lem:ModulusK0Realsable}
  A function $F \colon \Nat^{\Nat} \to \Nat$  is $K_0$-realisable 
  if and only if $F$ has a pointwise continuous modulus of continuity.
\end{lemma}
\begin{proof}
  Suppose that $F$ is realised by $\gamma \in K_0$.
  By $\ACUni$, define $g \colon \Nat^{\Nat} \to \Nat$  by
  \[
    g(\alpha) \defeql \min\left\{ n \in \Nat \mid
      \gamma(\overline{\alpha}n) > 0 \right\}.
  \]
  Then, $g$ is a modulus of $F$. It is also clear that $g$ is pointwise continuous.

  Conversely, suppose that $F$ has a pointwise continuous modulus 
  $g \colon \Nat^{\Nat} \to \Nat$. 
  Define a function $\gamma \colon \Nat^{*} \to \Nat$ by
  \[
  \gamma(a) \defeql \begin{cases}
    F(a*0^{\omega}) + 1&
       \text{if $\left(\exists a' \preccurlyeq a \right)|a'| \geq g(a' * 0^{\omega})$},\\
    0& \text{otherwise}.
  \end{cases}
  \]
  We show that $\gamma \in K_0$  and that $\gamma$ realises $F$.
  Let $\alpha \in \Nat^{\Nat}$. Since $g$ is pointwise continuous, there exists
  $n \in \Nat$ such that $n \geq g(\overline{\alpha}n * 0^{\omega})$.
  Since $g$ is a modulus of $F$,
  \[
    \gamma(\overline{\alpha}n) = F(\overline{\alpha}n * 0^{\omega}) +
    1 = F(\alpha) + 1.
  \]
  Next, let $a \in \Nat^{*}$ and suppose that $\gamma(a) > 0$.
  Then, there exists $a' \preccurlyeq a$ such that $|a'| \geq
  g(a' * 0^{\omega})$. Thus, $F(a' * 0^{\omega}) = F(a*0^{\omega}) =
  F(a * b * 0^{\omega})$ for all $b \in \Nat^{*}$. Hence, $\left(
  \forall b \in \Nat^{*} \right) \gamma(a) = \gamma(a * b)$.
  Therefore, $\gamma \in K_0$ and  $\gamma$ realises $F$.
\end{proof}

We recall the following result from Troelstra and van Dalen \cite[Proposition 8.13
(i)]{ConstMathI}.
\begin{lemma}\label{lem:KK0equivdBI}
  $\dBI \iff K = K_0$.
\end{lemma}
\begin{proof}
  See Troelstra and van Dalen \cite[Proposition 8.13 (i)]{ConstMathI}.
\end{proof}

\begin{proposition}\label{prop:dBIUC}
  The following are equivalent.
  \begin{enumerate}
    \item\label{prop:dBIUC1} $\dBI$.
    \item\label{prop:dBIUC2} Every $K_{0}$-realisable function $F \colon \Nat^{\Nat} \to \Nat$
      is $K$-realisable.
    \item\label{prop:dBIUC3} Every function $F \colon \Nat^{\Nat} \to \Nat$
      that has a pointwise continuous modulus is $K$-realisable.
  \end{enumerate}
\end{proposition}
\begin{proof}
  In view of Lemma \ref{lem:ModulusK0Realsable} and Lemma
  \ref{lem:KK0equivdBI}, it suffices to show that \eqref{prop:dBIUC2}
  implies $K_0 \subseteq K$.

  Assume \eqref{prop:dBIUC2}. Let $\gamma \in K_0$.
  Define a neighbourhood function $\gamma' \in K_0$ by
  \[
    \gamma'(a) \defeql \begin{cases}
      0 &\text{if $\left(\forall b \preccurlyeq a\right) \gamma(b) = 0$},\\
      \min\left\{ |b| \mid b \preccurlyeq a \wedge \gamma(b) > 0
      \right\} + 1
      &\text{otherwise}.
    \end{cases}
  \]
  By the assumption, there exists a Brouwer-operation $\xi \in K$ that realises the function 
  $F_{\gamma'} \colon \Nat^{\Nat} \to \Nat$ induced by $\gamma'$.
  By Lemma \ref{lem:keylem}, we may assume that
  \[
    \left( \forall a  \in \Nat^{*} \right) \xi(a) > 0  \imp
    |a| > \xi(a).
  \]
  Let $a \in \Nat^{*}$, and suppose that $\xi(a) > 0$. Then,
    $
    |a| > \xi(a) = F_{\gamma'}(a * 0^{\omega}) + 1.
    $
  Thus, there exists $k \in \Nat$ such that
    $
    \gamma'(\overline{a * 0^{\omega}}k) = F_{\gamma'}(a * 0^{\omega}) + 1.
    $
  By the definition of $\gamma'$, there exists $b \preccurlyeq
  \overline{a * 0^{\omega}}k$ such that $\gamma(b) > 0$ and $|b| + 1 =
  \gamma'(\overline{a * 0^{\omega}}k)$. Hence $b \preccurlyeq
  a$ so that $\gamma(a) > 0$.  By Lemma \ref{lem:KK0}, we obtain
  $\gamma \in K$.
\end{proof}

\begin{remark}\label{rem:BergerdFTUC}
  The decidable fan theorem is a version of the fan theorem formulated
  with respect to decidable bars on $\Bin^{*}$.  Berger
  \cite{BergerFANandUC} showed  that the decidable fan theorem and
  the following statement are equivalent:
  \begin{quote}
    Every function $F \colon \Bin^{\Nat} \to \Nat$ that has a pointwise
    continuous modulus is uniformly continuous.
  \end{quote}
  Here, a modulus of $F \colon \Bin^{\Nat} \to \Nat$
  is similarly defined as in \eqref{eq:modulus}.
  Proposition \ref{prop:dBIUC} says that this characterisation
  naturally extends to the decidable bar induction.
\end{remark}

\subsection{Monotone bar induction}
The monotone bar induction $\mBI$ is the following statement:
\begin{description}
  \item[\mBI] For any monotone bar $P \subseteq \Nat^{*}$ and a predicate $Q \subseteq
    \Nat^{*}$, if $P \subseteq Q$ and $Q$ is inductive, then
    $Q(\langle \rangle)$.
\end{description}
Here, a bar $P \subseteq \Nat^{*}$ is \emph{monotone} if
  $
  \left( \forall a,b \in \Nat^{*} \right) P(a) \imp P(a*b).
  $

A predicate $R \subseteq \Nat^{\Nat} \times \Nat$
is said to be \emph{locally continuous} if
\begin{equation*}
  \left( \forall \alpha \in  \Nat^{\Nat}\right)
  \left( \exists x \in \Nat \right)
  \left( \exists y \in \Nat \right)
  \left( \forall \beta \in \overline{\alpha}x \right)
   R(\beta,y).
\end{equation*}
Given a locally continuous predicate $R \subseteq \Nat^{\Nat} \times \Nat$,
we say that a function $F \colon \Nat^{\Nat} \to \Nat$ \emph{refines}
$R$ if $\left( \forall \alpha \in \Nat^{\Nat} \right) R(\alpha,
F(\alpha))$, i.e.\ $F$ is a choice function of $R$.

\begin{proposition}\label{prop:mBILUC}
  The following are equivalent.
  \begin{enumerate}
    \item\label{prop:mBILUC1} $\mBI$.
    \item\label{prop:mBILUC2} Every locally continuous predicate $R
      \subseteq \Nat^{\Nat} \times \Nat$ has a $K$-realisable function
      that refines $R$.
  \end{enumerate}
\end{proposition}
\begin{proof}
  \eqref{prop:mBILUC1} $\Rightarrow$ \eqref{prop:mBILUC2}
  Assume $\mBI$. Let $R \subseteq \Nat^{\Nat} \times \Nat$ be a
  locally continuous predicate. Define a predicate $P \subseteq \Nat^{*}$ by
  \[
    P(a) \defeqiv \left( \exists x \in \Nat \right)\left( \forall
    \alpha \in \Nat^{\Nat} \right) \alpha \in a \imp R(\alpha, x).
  \]
  Clearly, $P$ is a monotone bar. Define a predicate $Q \subseteq
  \Nat^{*}$  by
  \begin{multline}
    \label{eq:mBLUC21_Q}
    Q(a) \defeqiv \left( \exists \gamma \in K \right)
    \left( \forall \alpha \in \Nat^{\Nat} \right)
    \left( \forall b \in \Nat^{*} \right)  \\
    \bigl[ \gamma(b) > 0 \wedge \alpha \in a * b \bigr] \imp R(\alpha, \gamma(b)
    \dotminus 1).
  \end{multline}
  We show that
  \begin{enumerate}
    \item\label{prop:PQ_mBILUC} $P \subseteq Q$,
    \item\label{prop:QIND_mBILUC} $Q$ is inductive.
  \end{enumerate}

  \eqref{prop:PQ_mBILUC}  Let $a \in \Nat^{*}$ such that $P(a)$.
  Then, there exists $n \in \Nat$ such that $\left( \forall \alpha \in
  \Nat^{\Nat} \right) \alpha \in a \imp R(\alpha, n)$. Put $\gamma \defeql \lambda a. n + 1$,
  which is in $K$.
  Then, $\gamma$ is a witness of the existential quantifier in
  \eqref{eq:mBLUC21_Q}.  Thus $Q(a)$.

  \eqref{prop:QIND_mBILUC} 
  Let $a \in \Nat^{*}$ and suppose that $\left( \forall n \in \Nat
  \right) Q(\Cons{a}{n})$.  By $\AC$, there exists a sequence 
  $\left( \gamma_{n} \right)_{n \in \Nat}$ of Brouwer-operations
  such that
  \[
    \left( \forall n \in \Nat \right)
    \left( \forall \alpha \in \Nat^{\Nat} \right)
    \left( \forall b \in \Nat^{*} \right)
    \left[ \gamma_{n}(b) > 0 \wedge \alpha \in \Cons{a}{n} * b \imp
    R(\alpha, \gamma_{n}(b) \dotminus 1)\right].
  \]
  Put $\gamma \defeql \sup_{n \in \Nat}\gamma_{n}$.  Let
  $\alpha \in \Nat^{\Nat}$ and  $b \in
  \Nat^{*}$, and suppose that $\gamma(b) > 0$ and $\alpha \in a * b$.
  Then, there exist $n \in \Nat$ and $b' \in \Nat^{*}$ such that
  $b = \langle n \rangle * b' \wedge \gamma_{n}(b') > 0$.  Thus, $\alpha \in \Cons{a}{n} * b'$, so
  $R(\alpha, \gamma_{n}(b') \dotminus 1)$, that is
  $R(\alpha, \gamma(b) \dotminus 1)$. Hence $Q(a)$.

  By $\mBI$, we obtain $Q(\nil)$, i.e.\ there exists a Brouwer-operation $\gamma \in K$
  such that
  \[
    \left( \forall \alpha \in \Nat^{\Nat} \right)\left( \forall a \in \Nat^{*} \right)
    \gamma(a) > 0 \wedge \alpha \in a \imp 
    R(\alpha, \gamma(a) \dotminus 1).
  \]
  Thus,  the function $F_{\gamma} \colon
  \Nat^{\Nat} \to \Nat$ induced by $\gamma$ refines $R$.

  \eqref{prop:mBILUC2} $\Rightarrow$ \eqref{prop:mBILUC1}
  Assume \eqref{prop:mBILUC2}. Let $P$ be a monotone bar, and
  let $Q \subseteq \Nat^{*}$ be an inductive predicate such that
  $P \subseteq Q$. Define a predicate $R \subseteq \Nat^{\Nat} \times
  \Nat$ by
  \[
    R(\alpha, x) \defeqiv P(\overline{\alpha}x). 
  \]
  Then $R$ is clearly locally continuous.
  Thus, there exists a Brouwer-operation $\gamma \in K$ such that
  \begin{equation*}
    \left( \forall \alpha \in \Nat^{\Nat} \right)
    P(\overline{\alpha}F_{\gamma}(\alpha)).
  \end{equation*}
  By Lemma \ref{lem:keylem}, we may assume that
    $
    \left( \forall a \in \Nat^{*} \right) \gamma(a) > 0  \imp
    |a| > \gamma(a).
    $
  Let $a \in \Nat^{*}$ such that $\gamma(a) > 0$.
  Then, we have $\overline{a * 0^{\omega}}\gamma(a) \preccurlyeq
  \overline{a * 0^{\omega}}|a| = a$. Since
  $P(\overline{a * 0^{\omega}}(\gamma(a) \dotminus 1))$ and $P$ is monotone,
  we have $P(a)$, and thus $Q(a)$. Since $Q$ is inductive,
  we obtain $Q(\nil)$ by Proposition \ref{prop:UnSec}.
\end{proof} 
\section{Continuity axioms}\label{sec:PCN}
A continuity axiom states that if we have $\left(\forall \alpha \in
\Nat^{\Nat}\right)\left( \exists x \in \Nat \right) R(\alpha,x)$ , then
the dependence of $x \in \Nat$ on $\alpha \in \Nat^{\Nat}$
is continuous. By varying the strength of continuity with which
$x$ depends on $\alpha$, we obtain several principles.
The following continuity axioms are well known; see Troelstra and van
Dalen \cite[Chapter 4, Section 6 and Section 8]{ConstMathI}.
\begin{description}
  \item[\BCN] $\left( \forall \alpha \in \Nat^{\Nat} \right)
    \left( \exists x \in \Nat \right) R(\alpha,x) \imp
    \left( \exists \gamma \in K \right)
    \left( \forall \alpha \in \Nat^{\Nat} \right)
    R(\alpha,F_{\gamma}(\alpha))$.

  \item[\CN] $\left( \forall \alpha \in \Nat^{\Nat} \right)
    \left( \exists x \in \Nat \right) R(\alpha,x) \imp 
    \left( \exists \gamma \in K_{0} \right)
    \left( \forall \alpha \in \Nat^{\Nat} \right)
    R(\alpha,F_{\gamma}(\alpha))$.

  \item[\WCN] $\left( \forall \alpha \in \Nat^{\Nat} \right)
    \left( \exists x \in \Nat \right) R(\alpha,x) \imp 
    \left( \forall \alpha \in \Nat^{\Nat} \right)
    \left( \exists x,y \in \Nat \right)
    \left( \forall \beta \in \overline{\alpha}x \right) R(\beta,y)$.
\end{description}
Here, $F_{\gamma}$ is the function $F_{\gamma} \colon \Nat^{\Nat} \to \Nat$
induced by $\gamma \in K$ (or $\gamma \in K_{0}$).
The notions of continuity that correspond to $\BCN$,
$\CN$, and $\WCN$ are that of $K$-realisability, $K_0$-realisability, and local
continuity respectively.

The following is immediate from Proposition \ref{prop:dBIUC} and
Proposition \ref{prop:mBILUC}.
\begin{theorem}\label{thm:BarCont}
  \leavevmode
  \begin{enumerate}
    \item\label{thm:BarCont1}  $\BCN \iff \dBI + \CN$.
    \item\label{thm:BarCont2}  $\BCN \iff \mBI + \WCN$.
  \end{enumerate}
\end{theorem}

\begin{remark}
  Theorem \ref{thm:BarCont} is not new.
  The equivalence \eqref{thm:BarCont1}
  can be found in Troelstra and van Dalen \cite[Chapter 4, Proposition 8.13
  (iii)]{ConstMathI}, and the equivalence
  \eqref{thm:BarCont2} was shown by Kreisel and Troelstra \cite[Theorem 5.6.3
  (ii)]{KreiselTroelstra}.
  However, Proposition \ref{prop:dBIUC} and Proposition \ref{prop:mBILUC}
  make these equivalences obvious. Moreover, they clarify the
  complementary roles of various versions of bar induction and continuity
  axiom, which is one of the main contributions of the present work.
\end{remark}

We can formulate a continuity axiom with respect to the notion
of pointwise continuity. The principle of pointwise continuity (\PCN) is the following
statement:
\begin{description}
  \item[\PCN] $\left( \forall \alpha \in \Nat^{\Nat} \right)
              \left( \exists x \in \Nat \right) R(\alpha,x) \\\imp
              \left( \exists \delta \in  \Nat^{\Nat^{*}} \right)
              \left( \forall \alpha \in \Nat^{\Nat} \right) \left(
              \exists x \in \Nat \right) \left( \forall a \in
              \Nat^{*} \right) \delta(\overline{\alpha}x) =
              \delta(\overline{\alpha}x * a) \wedge 
              R(\alpha, \delta(\overline{\alpha}x))$.
\end{description}
The principle $\PCN$ asserts the existence of a pointwise continuous choice
function from the assumption $\left( \forall \alpha \in \Nat^{\Nat} \right)
              \left( \exists x \in \Nat \right) R(\alpha,x)$.
One can show that $\PCN$ is equivalent to the following statement:
\begin{multline*}
 \left( \forall \alpha \in \Nat^{\Nat} \right)
  \left( \exists x \in \Nat \right) R(\alpha,x) \\ \imp
  \left( \exists \delta \in  \Nat^{\Nat^{*}} \right)
  \left( \forall \alpha \in \Nat^{\Nat} \right) \left(
  \exists x \in \Nat \right) \left( \forall a \in
  \Nat^{*}
  \right) \delta(\overline{\alpha}x) =
  \delta(\overline{\alpha}x * a)\\ \wedge \left( \forall
  \beta \in \overline{\alpha}x
  \right) R(\beta, \delta(\overline{\alpha}x)).
\end{multline*}
The following equivalence is immediate from Theorem
\ref{thm:EquivcBIUCb}.
\begin{proposition}
  $\BCN \iff \cBI + \PCN$.
\end{proposition}

\section{$\Pi^{0}_{1}$ bar induction}
The $\Pi^{0}_{1}$ bar induction ($\PiOneBI$) is defined with
respect to a bar that is a $\Pi^{0}_{1}$-set, where a predicate $P
\subseteq \Nat^{*}$ is a \emph{$\Pi_{1}^{0}$-set} if there is a
detachable predicate $D \subseteq \Nat^{*} \times \Nat$ such that
\[
  \left( \forall a \in \Nat^{*} \right) \left[ P(a) \leftrightarrow
  \left( \forall n \in \Nat \right) D(a,n)\right].
\]
Specifically, $\PiOneBI$ is the following statement:
\begin{description}
  \item[\PiOneBI] For any $\Pi^{0}_{1}$-bar $P \subseteq \Nat^{*}$ and a
    predicate $Q \subseteq \Nat^{*}$, if $P \subseteq Q$ and $Q$ is
    inductive, then $Q(\langle \rangle)$.
\end{description}
Note that every c--bar is a $\Pi^{0}_{1}$-set modulo the coding of finite
sequences in $\Nat$. Thus, $\PiOneBI$ implies $\cBI$.
We show, however, that $\PiOneBI$ is not an intuitionistic principle.

Recall that $\LLPO$ (the lesser limited principle of omniscience) is 
$\Sigma^{0}_{1}$ De Morgan's Law, i.e.\ for any $\alpha, \beta \in
\Nat^{\Nat}$,
\begin{multline*}
  \neg \left[ \left( \exists n \in \Nat \right) \alpha(n) \ne 0 \wedge \left( 
  \exists n \in \Nat\right) \beta(n) \ne 0 \right] \imp \\
  \neg \left(\exists n \in \Nat  \right)\alpha(n) \ne 0 \vee
  \neg \left(\exists n \in \Nat  \right) \beta(n) \ne 0.
\end{multline*}
\begin{proposition}
$\PiOneBI$ implies $\LLPO$.
\end{proposition}
\begin{proof}
  Assume $\PiOneBI$.  Let $\alpha,\beta \in \Nat^{\Nat}$, and suppose that
  \[
  \neg \left[ \left(\exists n  \in \Nat  \right) \alpha(n) \ne 0
    \wedge \left(\exists n \in \Nat  \right) \beta(n) \ne
  0 \right].
  \]

  Define a predicate $P \subseteq \Nat^{*}$ by
  \[
    P \defeql \left\{ \langle n \rangle  \mid \alpha(n) = 0 \right\}
       \cup \left\{ \langle   \rangle \mid \left( \forall n \in \Nat \right)
       \beta(n) = 0 \right\}.
  \]
  Note that  $P$ is a
  $\Pi^{0}_{1}$-set.  We show that $P$ is a bar. Let $\gamma \in
  \Nat^{\Nat}$. Then, either $\alpha(\gamma(0)) = 0$ or
  $\alpha(\gamma(0)) \ne 0$.  If $\alpha(\gamma(0)) = 0$, then
  $\overline{\gamma} 1 \in P$.  If $\alpha(\gamma(0)) \ne 0$, then
  $\left( \exists n \in \Nat \right) \beta(n) \ne 0$ implies
  $
   \left(\exists n \in \Nat  \right)\alpha(n) \ne 0 \wedge \left( 
   \exists n \in \Nat\right) \beta(n) \ne 0$,
  a contradiction. Thus, $\left(\forall n \in \Nat \right)
  \beta(n) = 0$. Hence, $\overline{\gamma}0 = \langle   \rangle \in
  P$. Therefore, $P$ is a bar. 
  
  Define a predicate $Q \subseteq \Nat^{*}$ by
  \[
    Q \defeql P \cup \left\{ \langle \rangle \mid \left( \forall n
      \in \Nat \right) \alpha(n) = 0 \right\}.
  \]
  Then, $Q$ is clearly inductive and $P \subseteq Q$. Thus, by
  $\PiOneBI$, we have $\langle \rangle \in Q$, i.e.\
  \[
    \left( \forall n \in \Nat \right) \alpha(n) = 0 \vee
    \left( \forall n \in \Nat \right) \beta(n) = 0,
  \]
  or equivalently,
  $
  \neg \left( \exists n \in \Nat \right) \alpha(n) \ne 0 \vee
  \neg \left( \exists n \in \Nat \right) \beta(n) \ne 0.
  $
\end{proof}
It is well known that the $\Sigma^{0}_{1}$ bar induction implies $\LPO$
(the limited principle of omniscience, also known as the $\Sigma^{0}_{1}$ law
of excluded middle); see Troelstra and van Dalen \cite[Chapter 4,
Excercise 4.8.11]{ConstMathI}.\footnote{The example of the bar that is
used to derive $\LPO$ from the $\Sigma^{0}_{1}$ bar induction is
attributed to Kleene \cite[Section 7.14]{KleeneVesley},
but the bar defined in \cite[Section
7.14]{KleeneVesley} is not a $\Sigma^{0}_{1}$ set.} Since
the continuity axiom $\WCN$ contradicts $\LLPO$ (Troelstra and van
Dalen \cite[Chapter 4, Proposition 6.5]{ConstMathI}),
those facts show that the monotonicity of the bar is essential for
an intuitionistically acceptable formulation of bar induction.
Note that the situation is quite different for the fan theorem;
since the $\Pi_{1}^{0}$ fan theorem (the fan theorem with respect to
$\Pi_{1}^{0}$ binary bars) is an instance of the full fan theorem, it is
intuitionistically acceptable.

\section{Further work}
We now have the following implications.
  \begin{enumerate}
    \item $\mBI \implies \cBI \implies \dBI$.
    \item $\BCN \implies \CN \implies \PCN \implies \WCN$.
  \end{enumerate}
It remains to be seen which of these implications are
strict, that is cannot be reversed.
In view of the strength of the notion of
continuity associated with each principles, we conjecture that all
of the above implications are strict.

\section*{Acknowledgements}
We thank Hajime Ishihara and Takako Nemoto for helpful comments on the
subject of this paper. The author is supported by Core-to-Core Program
A.~Advanced Research Networks by Japan Society for the Promotion of
Science (JSPS).

\begin{thebibliography}{10}

\bibitem{BeesonFoundationConstMath}
M.~J. Beeson.
\newblock {\em Foundations of Constructive Mathematics}.
\newblock Springer, Berlin Heidelberg, 1985.

\bibitem{BergerFANandUC}
J.~Berger.
\newblock The fan theorem and uniform continuity.
\newblock In S.~B. Cooper, B.~L{\"o}we, and L.~Torenvliet, editors, {\em New
  Computational Paradigms. CiE 2005}, volume 3526 of {\em Lecture Notes in
  Comput.\ Sci.}, pages 18--22. Springer, Berlin, Heidelberg, 2005.

\bibitem{BergerUCandcFT}
J.~Berger.
\newblock The logical strength of the uniform continuity theorem.
\newblock In A.~Beckmann, U.~Berger, B.~L{\"o}we, and J.~V. Tucker, editors,
  {\em Logical Approaches to Computational Barriers. CiE 2006}, volume 3988 of
  {\em Lecture Notes in Comput.\ Sci.}, pages 35--39. Springer, Berlin,
  Heidelberg, 2006.

\bibitem{Bishop-67}
E.~Bishop.
\newblock {\em Foundations of Constructive Analysis}.
\newblock McGraw-Hill, New York, 1967.

\bibitem{HowardKreisel}
W.~A. Howard and G.~Kreisel.
\newblock Transfinite induction and bar induction of types zero and one, and
  the role of continuity in intuitionistic analysis.
\newblock {\em J.\ Symbolic Logic}, 31:325--358, 9 1966.

\bibitem{ConstRevMatheCompactness}
H.~Ishihara.
\newblock Constructive reverse mathematics: compactness properties.
\newblock In L.~Crosilla and P.~Schuster, editors, {\em From Sets and Types to
  Topology and Analysis: Towards Practicable Foundations for Constructive
  Mathematics}, number~48 in Oxford Logic Guides, pages 245--267. Oxford
  University Press, 2005.

\bibitem{KawaiFContOnBaire}
T.~Kawai.
\newblock Formally continuous functions on {B}aire space.
\newblock {\em Mathematical Logic Quarterly}, 64(3):192--200, 2018.

\bibitem{KleeneVesley}
S.~C. Kleene and R.~E. Vesley.
\newblock {\em The foundations of intuitionistic mathematics, especially in
  relation to recursive functions}.
\newblock North Holland, Amsterdam, 1965.

\bibitem{KreiselTroelstra}
G.~Kreisel and A.~S. Troelstra.
\newblock Formal systems for some branches of intuitionistic analysis.
\newblock {\em Annals of Mathematical Logic}, 1(3):229--387, 1970.

\bibitem{ConstMathI}
A.~S. Troelstra and D.~{van Dalen}.
\newblock {\em Constructivism in Mathematics: An Introduction. Volume {I}},
  volume 121 of {\em Studies in Logic and the Foundations of Mathematics}.
\newblock North-Holland, Amsterdam, 1988.

\bibitem{ConstMathII}
A.~S. Troelstra and D.~{van Dalen}.
\newblock {\em Constructivism in Mathematics: An Introduction. Volume {II}},
  volume 123 of {\em Studies in Logic and the Foundations of Mathematics}.
\newblock North-Holland, Amsterdam, 1988.

\end{thebibliography}
\newcommand{\noop}[1]{}

\end{document}